\documentclass{article}
\usepackage[utf8]{inputenc}
\usepackage{a4}
\usepackage{amssymb}
\usepackage{amsmath}
\usepackage{amsthm}
\usepackage{amstext}
\usepackage{amscd}
\usepackage{latexsym}
\usepackage{graphics}
\usepackage{color}
\usepackage{verbatim}
\usepackage{amsgen, amsfonts, euscript, enumerate, url, calc}
\usepackage{color}
\usepackage{graphicx}
\usepackage{amssymb,mathrsfs}
\usepackage{caption}
\usepackage{subcaption}
\usepackage{authblk}

\numberwithin{equation}{subsection}

\newtheorem{defn}{Definition}[section]
\newtheorem{thm}{Theorem}[section]
\newtheorem{cor}{Corollary}[section]
\newtheorem{lemma}{Lemma}[section]

\newtheorem{proposition}{Proposition}[section]

\usepackage[top=1.3 in,bottom=1 in,left=1.2 in,right= 1 in,footskip=0.4in]{geometry}

\title{Reconstruction of hypergraphs from line graphs and degree sequences }

\author[1]{Amitava Bhattacharya }
\author[2]{Aloysius Godinho } 
\author[3]{Pritam Majumder}
\author[4]{Navin  Singhi }
\affil[1,3,4]{School of Mathematics, Tata Institute of Fundamental Research, Mumbai, India\\ \texttt{\{amitava, pritam, singhi\}@math.tifr.res.in}}
\affil[2]{Department of Mathematics, 
Rosary College,
Goa, India\\
 \texttt{aloysius@rosarycollege.org}
}
\date{}

\begin{document}

\maketitle

\begin{abstract}
 In this paper we consider the problem  to reconstruct a $k$-uniform hypergraph  from its line graph. In general this problem is hard. We solve this problem when the number of hyperedges containing any pair of vertices is bounded.  Given an integer sequence, constructing a $k$-uniform hypergraph with that as its degree sequence is NP-complete. Here we  show that for constant integer sequences the question can be answered in polynomial time using Baranyai's theorem.
 \end{abstract}



\section{Introduction}

Let $X = \{x_1,x_2, ..., x_n\}$ be a finite set. A {\it hypergraph} on $X$ is a family
$E = \{E_1,E_2,\ldots , E_m\}$ of non-empty subsets of $X$. 
Elements of $X$ are
called the vertices, while those of $E$ are called the edges of $H$. 

A hypergraph $H=(X, E)$ is said to be $k$-{\it uniform} if  $E \subset \binom {X}  {k}$, the set
of all $k$-subsets of $X$, where $k \geq 2$. A hypergraph $H$ is said
to be {\it linear} if every pair of distinct vertices of $H$ is in at most one edge of $H$. A $2$-uniform linear hypergraph is
called a graph.

\begin{defn}
The line graph of a hypergraph $H =(X, E)$, denoted
by $L(H)$, is the graph where $V(L(H)) = E$ and $E(L(H))$ is the set of all unordered
pairs $\{e, e'\}$ of distinct elements of $E$ such that $e\cap e' \neq \phi$ in $H$. 
\end{defn}

We denote the set of line graphs of $k$-uniform hypergraphs and $k$-uniform linear hypergraphs  by $L_k$ and $L_k^l$ respectively.\par\medskip 

\noindent For a graph $G$, we shall denote its  vertex set by $V( G)$, while
the edge set will be denoted by $E(G)$. Degree of a vertex is the number of edges that contain that vertex. The degree of an edge $xy$  is defined to be the number of distinct triangles in $G$ containing the edge. The minimum edge degree in the graph is denoted by $\delta_e(G)$ or simply $\delta_e$. For $W \subset V(G)$, $N(W)$ denotes the subset of vertices in $G$ which are adjacent to every vertex in $W$. A clique in $G$ refers to both a set of pairwise adjacent vertices and the corresponding induced complete
subgraph. The size of a clique is the number of vertices in the clique. \par\bigskip

In this paper we consider the problem of reconstruction of hypergraphs when partial information about its structure is given.
For simple graphs these questions  are well studied. 

If the degree sequence of the graph (degree of every vertex as a tuple) is specified then there are well known characterizations.
In particular the
Havel-Hakimi algorithm (\cite{havel, hakimi}) or Erd\"os Gallai inequalites (\cite{erdos_gallai}) give a good characterization. 
In the book \cite{book_peled} more than eight different characterizations are given.
For hypergraphs this problem is very hard. In \cite{onn,onn1} it was shown that this problem is NP-Complete by reducing it to a problem related to the {\it 3-partition problem}. It is unlikely that a `good' characterization for this problem exists.
In Section 3, we give a solution for characterizing  the degree sequence of  hypergraphs when all the degrees are constant. The solution uses an integral version of max-flow min-cut theorem.

\noindent The problem of characterizing the class of graphs $L_k$ has been studied for a very long time. Beineke \cite{beineke} in his classical work on line graphs characterized the class $L_2^l$ by a finite list of forbidden
induced subgraphs (finite characterization). This was later expanded upon by Bermond and Meyer \cite{meyer} who obtained a finite characterization for the class $L_2$ (intersection graphs of multigraphs). Lov\'asz \cite{lovasz} showed that for $k \geq 3$ the class $L_k^l$ has no finite characterization. Niak \textit{et al.} \cite{niak} obtained a finite characterization for the set of graphs in $L_3^l$ with $\delta \geq 69$., where $\delta$ represents the minimum vertex degree in a graph. This was further improved by Skums \textit{et al.} \cite{skums} who obtained a finite characterization for graphs in $L_3^l$ with $\delta \geq 16$. Metelsky \cite{Metelsky} proved that for $k \geq 4$ and any positive integer $a$, the set of graphs in $L_k^l$ with $\delta \geq a$ has no finite characterization.   \par\medskip

\noindent For a hypergraph $H=(X,E)$ and $z \in X$, the degree $d_H(z)$ of $z$ is defined to be the number of
edges of $H$ containing $z$, 
the maximum degree of the hypergraph $H$ is denoted by
$\Delta(H) = \max\limits_{z \in X} d_H(z)$. 
Similarly, for a pair of vertices $\{x,y\} \subset X$, we define the pair degree $d_H(\{x,y\})$ to be the number of edges in $H$ containing the pair $\{x,y\}$. We denote the maximum pair degree in $H$ by $\Delta_2(H)=\max_{\{ x,y\} \subset X} d_H(\{ x,y\})$. $\Delta_2(H)$ is called the \textit{multiplicity} of the hypergraph $H$. A hypergraph is linear if $\Delta_2(H)=1$.  Denote the set of intersection graphs of $k$-uniform hypergraphs with $\Delta_2(H)\leq p$ by $L_k^{(p)}$. Observe that for $p>2$, $L_k^l\subset L_k^{(2)} \subset \cdot \cdot \cdot \subset L_k^{(p)} $. 

In Section 2 we prove the following main theorem: 

\begin{thm}\label{main}
There is a polynomial $f(k,p)$ of degree at most 4 with the property that,
given any pair $k,p$, there exists a finite family $\mathcal{F}(k,p)$ of forbidden graphs such that any graph $G$
with minimum edge-degree at least $f(k,p)$ is a member of $L_k^{(p)}$ if and only if $G$ has no
induced subgraph isomorphic to a member of $\mathcal{F}(k,p)$.
\end{thm}

\section{Reconstructing Hypergraphs from line graphs}
\begin{lemma}
If $G \in L_k^{(p)}$ then the $G$ does not contain a $k+1$ claw.
\end{lemma}
\begin{proof}
Let $H=(X,E)$ be a $k$-uniform hypergraph with $\Delta_2(H) \leq p$ such that $G=L(H)$. Let $\langle x;y_1,y_2,\ldots ,y_{r} \rangle$ be a claw in $G$. If $x=\{x_1,x_2,\ldots ,x_k \} \in E$, then $x \cap y_i \neq \phi$ for every $i=1,2,\ldots , r$. Since each distinct pair $y_i,y_j$ are non-adjacent in $G$, $y_i \cap y_j=\phi$ in $E$. Therefore it follows that $r \leq k$.
\end{proof}

Figure~\ref{pic_claw1} and Figure~\ref{pic_claw3} illustrates the above lemma. Note that the hyperedge $0$ has $k+1$ vertices and hence the hypergraph fails to be $k$-uniform.

\begin{figure}[h]
\centering
\includegraphics[scale=1.0]{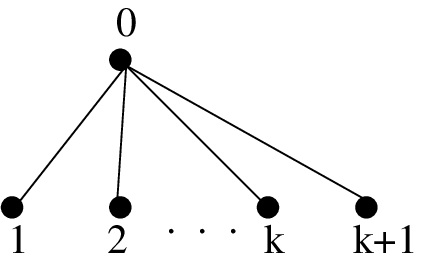}
\caption{}
\label{pic_claw1}
\end{figure}%
\begin{figure}[h]
\centering
\includegraphics[scale=0.8]{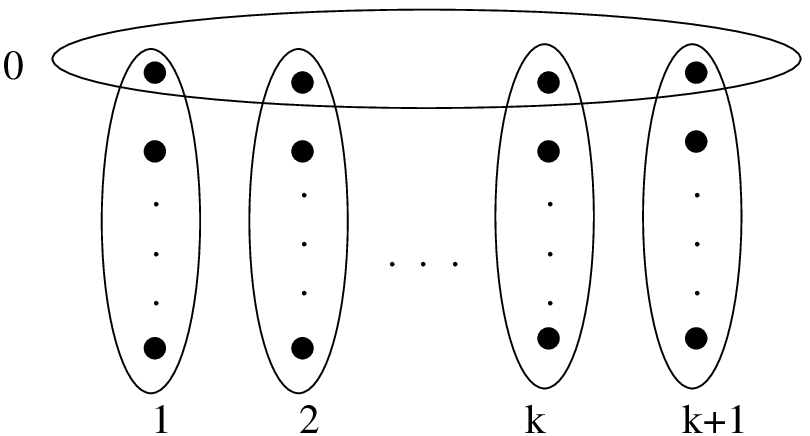}
\caption{}
\label{pic_claw3}
\end{figure}

\begin{lemma}\label{nonadj}
If $G \in L_k^{(p)}$ and $a,b \in V(G)$ such that $ab \not\in E(G)$, then $|N(\{a,b\})|\leq pk^2$ (where $N(\{a,b\}$ denotes the set of vertices adjacent to both $a$ and $b$).
\end{lemma}
\begin{proof}
Let $H=(X,E)$ be a $k$-uniform hypergraph with $\Delta_2(H) \leq p$ such that $G=L(H)$. Let $a=\{a_1,a_2,\ldots ,a_k \}, b=\{b_1 ,b_2 ,\ldots , b_k \} \in E$ and $a\cap b=\phi $. Now for every $v\in N(\{a,b\})$, $v\cap a, v\cap b \neq \phi$ in $E$. Therefore there is an $a_i \in a$ and $b_j \in b$ such that the pair $ \{a_i,b_j\} \subset x$ in $E$. The number of such pairs is $k^2$ and given that each pair can appear in at most $p$ edges, the result follows. See Figure~\ref{forb1} and Figure~\ref{forb11}.
\end{proof}

\begin{figure}[h]
\centering
\includegraphics[scale=0.8]{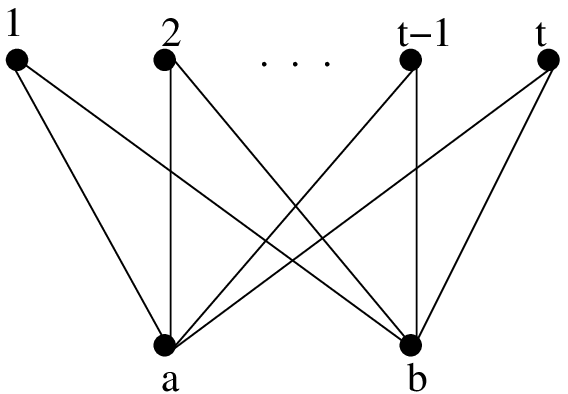}
\caption{}
\label{forb1}
\end{figure}%
\begin{figure}[h]
\centering
\includegraphics[scale=0.7]{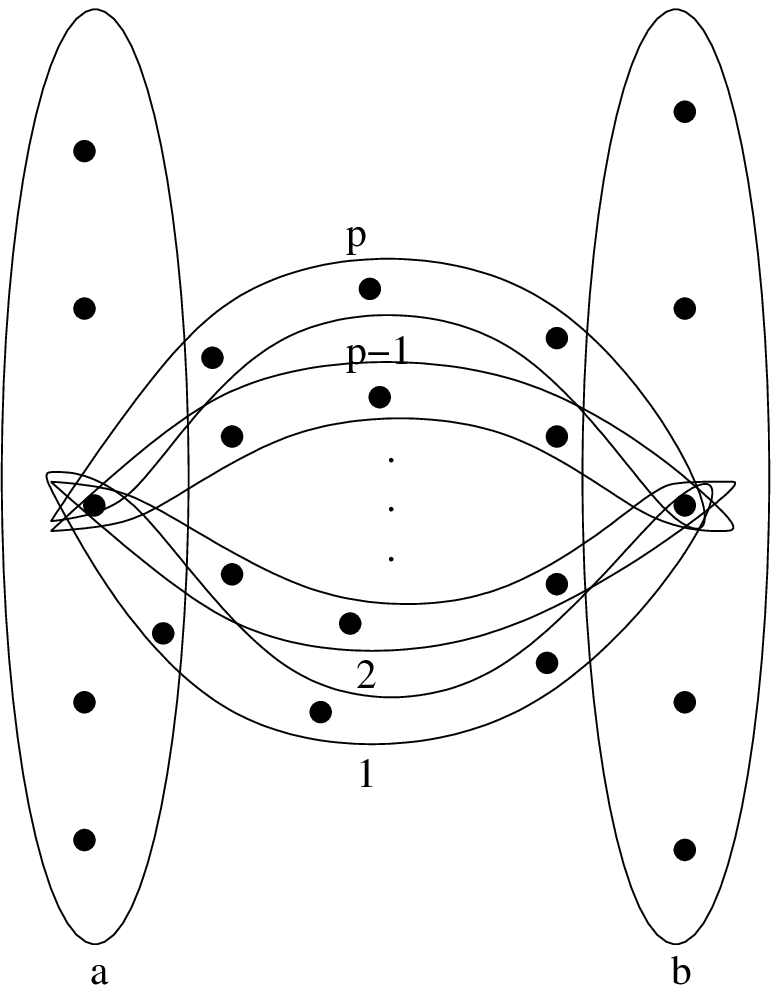}
\caption{}
\label{forb11}
\end{figure}

Next we prove the following lemma, which says that  ``large cliques'' (Figure~\ref{lg_clique_[1]})
in the line graph come from structures like Figure~\ref{lg_clique_[2]} in the hypergraph.

\begin{lemma}\label{uniqie}
Let $G \in L_k^{(p)}$ such that $G=L(H)$, $H=(X,E)$. If  $K$ is a clique of size at least $pk^2+(p-2)k+2$ in $G$, 
then there is an $x \in X$ such that $x \in v $ for every $v \in K$. 
\end{lemma}
\begin{proof}
Let $x=\{ x_1,x_2,\ldots ,x_k \} \in E$ be a vertex in $K$. Since $K$ is a clique, every vertex in $K$ 
contains at least one element of $x$. Now using a pigeon hole argument, given that $|V(K)|\geq pk^2+(p-2)k+2$ and $p \geq 1$, 
it follows that there is an $x_t \in x$ such that $x_t$ is in at least $kp+1$ vertices of $K$ (including $x$). 
Let $L=\{v_1,v_2,\ldots v_{kp+1}\}$ be a set of $kp+1$ vertices such that $x_t \in v_i$ for every $1\leq i \leq kp+1$. 
Let $y = \{y_1,y_2,\ldots ,y_k\}  \in V(K) \setminus  L$ such that $x_t \not\in y$. 
Since $y$ is adjacent to every vertex in $L$, $y\cap v_i \neq \phi$ for every $i$. 
Now for each $i$, the pair $x_t,y_i$ can appear in at most $p$ edges and the number of such pairs is $k$.
Hence it follows that a pair $x_t,y_l$ appears at least $p+1$ vertices of $L$.   
\end{proof}
\begin{figure}[h]
\centering
\begin{subfigure}{.5\textwidth}
\centering
\includegraphics[scale=0.8]{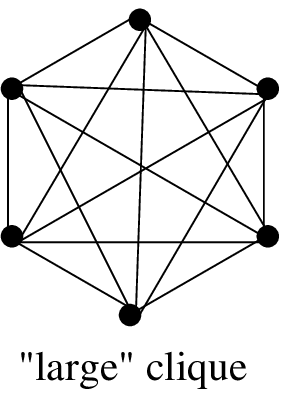}
\caption{}
\label{lg_clique_[1]}
\end{subfigure}%
\begin{subfigure}{.5\textwidth}
\centering
\includegraphics[scale=0.5]{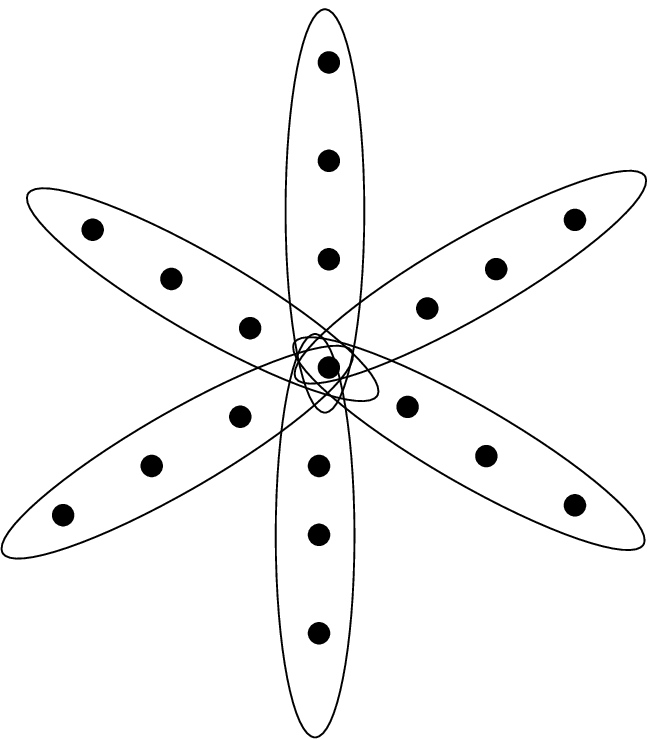}
\caption{}
\label{lg_clique_[2]}
\end{subfigure}
\caption{}
\label{lg_clique}
\end{figure}

\begin{lemma}\label{commonality}
Let $G \in L_k^{(p)}$ and $K$ be a maximal clique of size at least $pk^2+(p-2)k+2$. If $v \in V(G)$  
such that $v\not\in K$, then $v$ can be adjacent to at most $pk$ vertices in $K$.
\end{lemma}

Figure~\ref{forb2} illustrates this lemma.

\begin{proof}
Let $G \in L_k^{(p)}$ such that $G=L(H)$, $H=(X,E)$. Since $|V(K)|\geq pk^2+(p-2)k+2$, 
by Lemma \ref{uniqie} there is an $x \in X$ such that $x\in v$ for every $v \in K$. 
Suppose $y \not\in K $ is adjacent to $pk+1$ vertices in $K$ then using a similar argument as in the previous
proof it follows that $x \in y$. This contradicts the maximality of $K$. 
\end{proof}

\begin{figure}[h]
\centering
\includegraphics[scale=0.6]{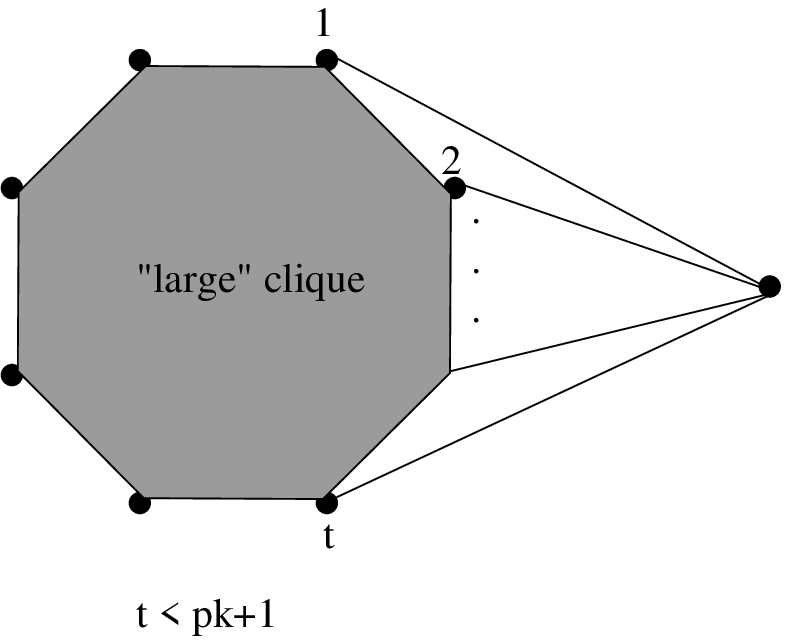}
\caption{}
\label{forb2}
\end{figure}

Next we define the following three forbidden families. 

\begin{enumerate}
\item $\mathcal{F}_1(p,k)$ denote the set of graphs $G$ of order $pk^2+3$ with two 
non-adjacent vertices $u,v$ such that $N(\{ v,w\}) =V(G) \setminus \{v,w \}$,
\medskip

\begin{figure}\centering
 \includegraphics[scale=0.7]{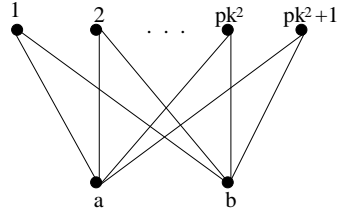}
 \caption{$\mathcal{F}_1(p,k)$}
\end{figure}

\item $\mathcal{F}_2(p,k)$ denote the set of graphs $G$ of order $pk^2+(p-2)k+3$ 
containing a maximal clique $K$ of size $pk^2+(p-2)k+2$ and a $v\not\in K$ such that $v$ is adjacent to 
at least $pk+1$  vertices of $K$,
\medskip
\begin{figure}[h]\centering
 \includegraphics[scale=0.6]{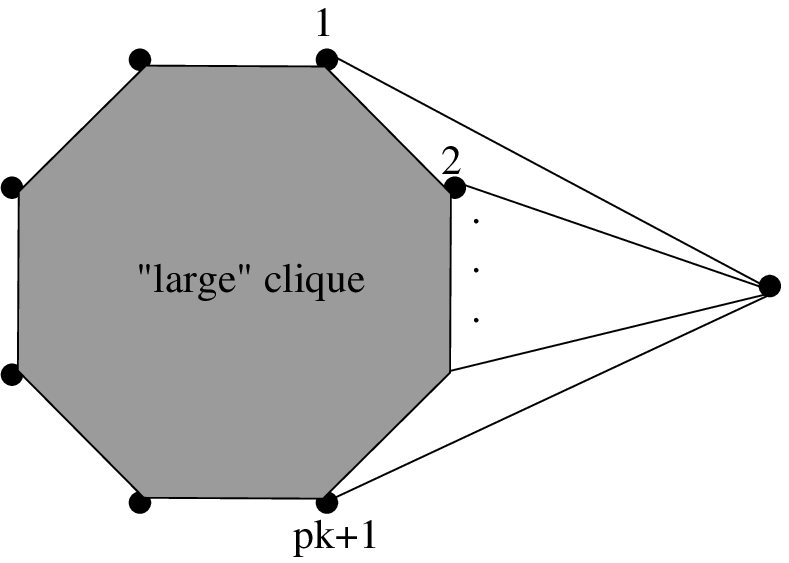}
\caption{ $\mathcal{F}_2(p,k)$}
 \end{figure}

\item $\mathcal{F}_3(p,k)$ denote the set of graphs $G$ of order less than $2(pk^2+(p-2)k+2)$ 
containing a  pair of distinct maximal cliques $K_1,K_2$ of size $pk^2+(p-2)k+2$ such that $|V(K_1)\cap V(K_2)| \geq p+1$.
\medskip

\begin{figure}[h]
\centering  \includegraphics[scale=0.65]{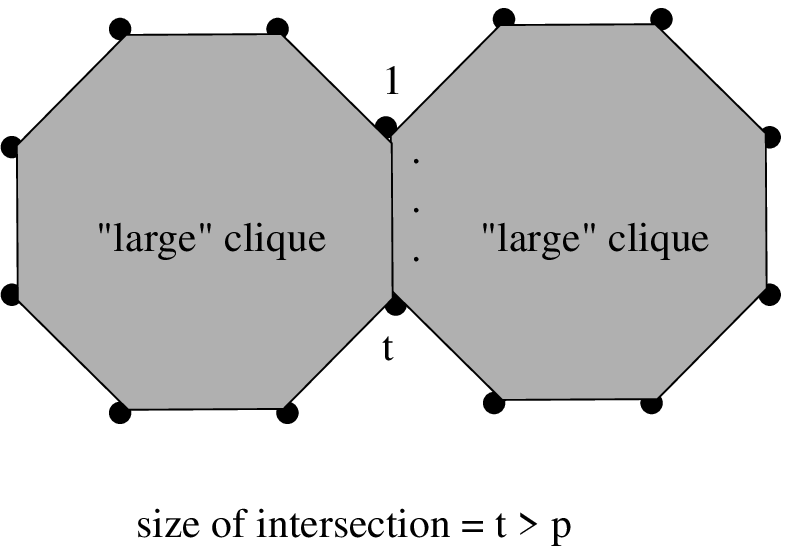}
\caption{ $\mathcal{F}_3(p,k)$}
\end{figure}
\end{enumerate}

Let $\mathcal{F}(p,k)=\mathcal{F}_1(p,k) \bigcup \mathcal{F}_2(p,k)
 \bigcup \mathcal{F}_3(p,k) \bigcup \{k+1\mbox{-claw}\}$.

\begin{lemma}\label{necessary}
If $G \in L_k^{(p)}$, then $G$ does not contain an induced subgraph from the set $\mathcal{F}(p,k)$. 
\end{lemma}
\begin{proof}
Suppose $G_1 \in \mathcal{F}_1(p,k)$ is an induced subgraph of $G$. Let  $x,y$ be 
non-adjacent vertices in $G_1$ such that 
$N_{G_1}(\{x,y \})=V(G_1)\setminus \{x,y \}.$ Then it 
follows that $$|N_G(\{x,y \})|\geq pk^2+1.$$ Since $xy \not\in E(G)$, this contradicts Lemma \ref{nonadj}. 

Let $G_2 \in \mathcal{F}_2(p,k)$ be an induced subgraph of $G$. Let $K$ be a maximal clique of size 
$pk^2+(p-2)k+2$ in $G_2$ and let $w \in V(G_2) \setminus K$ such that $w$ is adjacent to at least $pk+1$  
vertices of $K$. There exists a maximal clique $K'$ in $G$ such that $K \subset K'$. Since $|K'| \geq pk^2+(p-2)k+2$, 
by Lemma \ref{uniqie} there exists a $x \in X$ such that $x \in v$ for every $v \in K'$. 
Now since $w \not\in K$ and both $K$ and $K'$ are maximal cliques, it follows that $w \not\in K'$. Further $w$ is adjacent 
to at least $pk^2+1$ vertices of $K'$ this contradicts Lemma \ref{commonality}. 

Let $G_3 \in \mathcal{F}_3(p,k)$ be an induced subgraph of $G$. Let $K_1,K_2$ be maximal cliques in $G_3$ of 
order $pk^2+(p-2)k+2$ such that $|V(K_1)\cap V(K_2)| \geq kp+1$. There exists maximal cliques $K_1', K_2'$ in $G$ 
such that $K_1 \subset K_1'$ and $K_2 \subset K_2'$. By Lemma \ref{uniqie} there exists $x_1, x_2 \in X$ such 
that $x_1 \in v$ for every $v \in K_1$ and $x_2 \in v$ for every $v \in K_2$. Sine $K_1', K_2'$ are distinct maximal 
cliques in $G$ it follows that $x_1 \neq x_2$. Therefore if $|V(K_1') \cap V(K_2')| \geq |V(K_1) \cap V(K_2) )| \geq p+1$ 
it follows that the pair $x_1,x_2$ appears in at least $p+1$ distinct vertices in $G$ which is a 
contradiction (see the following figure). 
\end{proof}

\centerline{ \includegraphics[scale=0.65]{forbidden_3.eps}\hspace{10mm} \includegraphics[scale=0.65]{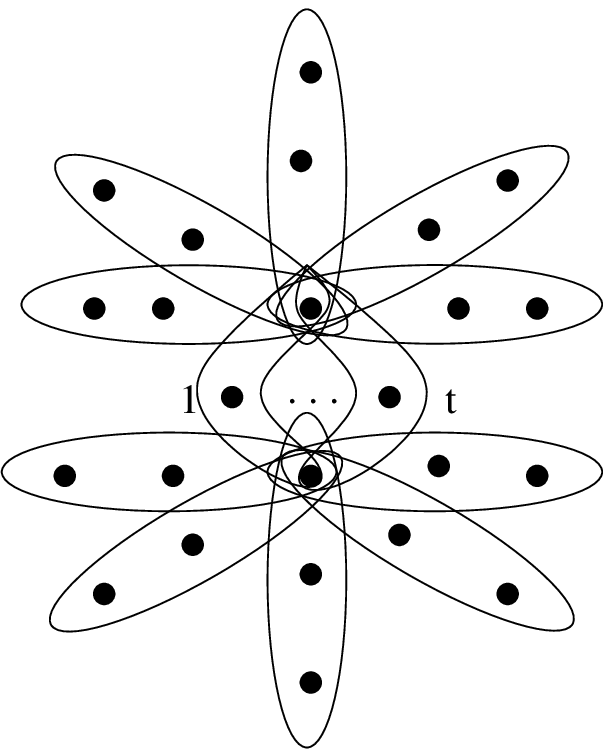}}

To show the other direction of the Theorem, we use the idea of ``clique partition'' for line graphs 
due to Krausz. We present a characterization of the members of the family $L_k^{(p)}$ which is a 
generalization of the criterion described in Berge \cite{berge}. 

\begin{proposition}\label{prop}
A graph $G \in L_k^{(p)}$ if and only if there is a collection of cliques  $\mathcal{K}:K_1,K_2,\ldots , K_r$ of $G$ such that;
\begin{enumerate}

    \item[(i)] Every edge belongs to at least one member of $\mathcal{K}$.
    
    \item[(ii)] Every vertex belongs to at most $k$ members of $\mathcal{K}$.
    
    \item[(iii)] If $K_i,K_j$ are distinct elements of  $\mathcal{K}$, then $|V(K_i) \cap V(K_j)| \leq p$.
\end{enumerate}

\end{proposition}

\begin{proof}
Let $G$ be a graph and $\mathcal{K}:K_1,K_2,\ldots ,K_r$ be a collection of cliques in $G$ 
satisfying (i)-(iii). For every $v \in V(G)$, let $g(v)$ denote the number of cliques in $\mathcal{K}$ 
containing the vertex $v$. From (ii) $g(v)\leq k$. Construct $\mathcal{K}'$ from $\mathcal{K}$ by 
adding $k-g(v)$ copies of $\{ v\}$, for 
every $v \in V$.  For $x \in V$, define $E_x=\{K \in \mathcal{K}'  \ | \ x \in K\} $. Let $H=(X,E)$ be the 
hypergraph with $X=\mathcal{K}'$ and $E=\{ E_x \ | \ x\in V\}$. Since each $E_x$ has cardinality $k$, it 
follows that $H$ is $k$-uniform. Suppose $ab$ is an edge in $G$, from (i) the edge $ab$ appears in at least one 
element of $\mathcal{K}$, therefore $E_a \cap E_b \neq \phi$. Finally for every distinct 
pair $K_1,K_2 \in\mathcal{K}'$, $|V(K_1) wcap V(K_2)| \leq p$, therefore the pair $K_1,K_2$ appear 
in at most $p$ edges of $H$. Hence $\Delta_2(H) \leq p$. \par \medskip

\noindent For the converse, let $H=(X,E)$ be a $k$-uniform hypergraph with $\Delta_2(H) \leq p$. 
For every $x \in X$ define $E_x = \{e\in E \ | \ x\in E \}$. Let $\mathcal{K}$ be the collection 
of all such sets $E_x$ which are non-empty. Let $G$ be the intersection graph of $H$. It is clear 
that $\mathcal{K}$ defines a collection of cliques in $G$. We shall show that this collection satisfies (i)-(iii).  
Let $e_1,e_2 \in E$  such that $e_1e_2$ is an edge in $G$. Hence $e_1 \cap e_2 \neq \phi$. Let $a \in  e_1 \cap e_2$, 
then the edge $e_1e_2$ appears in the clique $E_a \in \mathcal{K}$ this proves (i). 
To show (ii), let $x\in V(G)$. Let $\{x_1,\ldots ,x_k \} \in E$ represent the edge corresponding to $x$. 
Then the vertex $x$ appears in exactly $k$ cliques $E_{x_1}, \ldots ,E_{x_k} \in \mathcal{K}$. 
Finally if $e_1,e_2 \in E_a \cap E_b$, then ${a,b} \in e_1 \cap e_2$ and since the pair $a,b$ can 
occur in at most $p$ elements of $E$, it follows that $|E_a \cap E_b|\leq p$ this proves (iii). 
\end{proof}

Now we complete the proof of Theorem \ref{main}. For Lemmas \ref{lem1} - \ref{lem3} we
will assume that $G$ is a graph with at least one edge which has no induced subgraph
isomorphic to a member of $\mathcal{F}(p,k)$ and the minimum edge-degree of $G$ is at least $f(k,p) =pk^3+(p-3)k+1$. 
We then show that $G \in L_k^{(p)}$. Define $\mathcal{K}$ to be the set of all maximal cliques in $G$ of size at least $pk^2+(p-2)k+2$. 

\begin{lemma}\label{lem1}
Every edge in $G$ occurs in a clique of size at least $pk^2+(p-2)k+2$. 
\end{lemma}
\begin{proof}
Let $x,y$ be an edge in $G$. Let $\langle x;y,w_1,w_2,\ldots ,w_ r \rangle $ be a maximal claw at $x$ containing $y$. \par\medskip

\noindent \textbf{Case 1.} $r>0$: Let $\langle x;w_1,w_2,\ldots ,w_ r,v_1,v_2,\ldots ,v_s \rangle $ be a maximal claw at $x$ 
containing  $\langle x;w_1,w_2,\ldots ,w_ r \rangle $ as a subclaw. If one of $v_i$ is $y$ then $s=1$. Since $G$ does not 
contain a $k+1$-claw $r+s \leq k$. Now let $v_s=z$. Consider $N(\{x,z\})$. From our assumption about the minimum edge degree 
in $G$, $|N(\{x,z \})| \geq f(p,k)= pk^3+(p-3)k+1$. Now  some vertices in $N(\{x,z \})$ may be adjacent to vertices in 
the set $\{w_1,w_2,\ldots w_r,v_1,v_2\ldots v_{s-1}\}$ we shall discard these. Since $G$ does not contain any induced 
subgraph from the set $\mathcal{F}_1(p,k) $, it 
follows that $|N(\{z,w_i \})| \leq pk^2$ and $|N(\{z,v_j \})| \leq pk^2$, $1 \leq i \leq r, \ 1\leq j \leq s-1$. 
Hence there are at least $f(p,k)-(k-1)(pk^2-1)-1=pk^2+(p-2)k-1$ (if $z=y$) or $f(p,k)-(k-1)(pk^2-1)=pk^2+(p-2)k$ (if $z\neq y$) 
vertices in $N(\{ x,z \})$ which are not adjacent to any vertex in the set  $\{w_1,w_2,\ldots w_r,v_1,v_2\ldots v_{s-1}\}$. 
Since $\langle x;w_1,w_2,\ldots ,w_ r,v_1,v_2,\ldots ,v_s \rangle $ is a maximal claw, 
these vertices together with $x, y$ and $z$ must form a clique. Hence we have constructed a 
clique with at least $pk^2+(p-2)k+2$ vertices. \par\medskip

\noindent \textbf{Case 2.} $r=0$: Suppose that $\langle x;y \rangle $ is the maximum 
size claw at $x$ in this case $N(x)$ is a clique of size at least $pk^3+(p-3)k+1$. 
Since $p\geq 1, \ k\geq 2$, it follows that  $|N(x)|\geq pk^3+(p-3)k+1 \geq pk^2+(p-2)k+2 $, we are 
through. Otherwise let $\langle x;y_1,y_2,\ldots y_r \rangle $ be a maximal claw at $x$. 
Then repeating the same argument as in case 1 for the claw  $\langle x;y_1,y_2,\ldots y_r \rangle $ we 
obtain a clique of size at least $pk^2+(p-2)k+2$ containing the edge $xy$. 
\end{proof}

\begin{lemma}\label{join}
If $K  \in \mathcal{K}$ and $x \not\in K$, then $x$ is joined to at most $pk$ vertices of $K$.
\end{lemma}
\begin{proof}
If there is a vertex $x \in V(G)-V(K)$ such that $x$ is joined to at least $kp+1$ vertices in $K$ 
then $G$ will contain a member of $\mathcal{F}_2(p,k)$ as an induced subgraph. 
\end{proof}

\begin{lemma}\label{lem2}
Every vertex of $G$ is in at most $k$ distinct members of $\mathcal{K}$.
\end{lemma}
\begin{proof}
Suppose that the result is not true and let $x$ be a vertex of $G$ which is in
$k + 1$ distinct elements $K_1,\ldots ,K_{k+1}$ of $\mathcal{K}$. Now
let $a_1 \in K_1$ and $a_1 \neq  x$. By Lemma \ref{join}, it follows that $a_1$ is joined to at most $pk-1$ vertices
of $K_2$ other than $x$. Now since $|K_1 \cap K_2| \leq p $, 
there exists an $a_2 \in K_2$ such that $\langle x; a_1, a_2\rangle$ is a 2-claw. Suppose that we
have constructed an r-claw $\langle x; a_1, \ldots , a_r\rangle$, $r \leq k$, in $G$ such that $a_i \in  K_i$. Then each $a_i$
is joined to at most $pk-1$ vertices distinct from $x$ of $K_{r+ 1}$ and $\left|K_{r+1}\bigcap \left(\bigcup\limits_{i=1}^rK_i\right)\right| \leq rp$. Now $|K_{r+1}| \geq pk^2+(p-2)k+2 > r(kp-1)+r(p-1)+1$, hence  there exists an $a_{r+1} \in  K_{r+1}$ such that $\langle x; a_1, \ldots , a_{r+1}\rangle $  is an $(r + 1)$-claw in
$G$. Taking $r = k$, we get a $(k + 1)$-claw in $G$, contradicting the hypothesis.
\end{proof}

\begin{lemma}\label{lem3}
If $K_1, K_2  \in \mathcal{K}$ then $|V(K_1) \cap V(K_2)| \leq p$.
\end{lemma}
\begin{proof}
If there exist $K_1, K_2  \in \mathcal{K}$ such that $|V(K_1) \cap V(K_2)| \geq p+1$ then $G$ will 
contain a member of $\mathcal{F}_3(p,k)$ as an induced subgraph. 
\end{proof}

\noindent \textbf{Proof of Theorem 1.1}: The necessity that $G$ has no induced subgraph isomorphic
to $\mathcal{F}(p,k)$ follows from Lemma \ref{necessary}. The sufficiency that $G \in  L_k^{(p)}$
follows from Lemmas \ref{lem1}, \ref{lem2}, \ref{lem3} and Proposition \ref{prop}.

\section{Degree sequence of hypergraphs}
In this  Section, we consider the problem to characterize the degree sequence $(d_1,d_2,\ldots,d_n)$ of $k$-uniform hypergraphs. This problem was shown to be NP-complete in \cite{onn,onn1}. We consider the simple case when all the degrees are same. That is, when does $(d,d,\ldots, d)$ correspond to the degree sequence of a $k$-uniform hypergraph on $[N]$.  By double counting tuples $(i,E)$ where $i\in E$, and $E$ is an edge in
the hypergraph, it follows that $\sum_{i=1}^N d_i$ is divisible by  $k$.
We show using network flows that this necessary condition 
is also sufficient if the degrees are constant. This proof is  adapted from
the proof of Baranyai's theorem (\cite {baranyai}).

Let ${\cal P}:=\binom {[N]} k$ denote the sets of size $k$ of the set
$[N]$ and $N\ge k$. Let $M=\frac {k\binom {N} {k}}{L}$. We will partition $\cal P$ into $M$ disjoint collections of $k$ sets of same size. Let the collection be labelled as ${\cal S}_1,{\cal S}_2,\ldots ,{\cal S}_M$. Let $L=LCM (N,k)$. Number of $k$ sets in each class ${\cal S}_i$ is $\frac L k$ and each element $j\in [N]$ appears in a $k$-set in each
${\cal S}_i$ exactly $\frac L N$ times. Note that $M=\frac {k\binom {N} {k}}{L}$ and it is an integer.

 The existence of such a partition when $k\mid N$ is Baranyai's theorem and all known proofs of Baranyai's theorem use integrality of network flow (or an equivalent theorem).

We prove the existence of the required partition 
by induction on $\ell$ the number of alphabets. In particular we would show the following Theorem:
\begin{thm}\label{partition}
{\it For all $[\ell], \ell \le N$ there exists a collection $\{{\cal S}_1, {\cal S}_2,\ldots , {\cal S}_M \}$ 
of $\cal P$ such that a set $T\subset [\ell]$ appears $\binom {N-\ell} {k-|T|}$ times.}
\end{thm}
Note that  our desired partition corresponds to $\ell=N$.

\begin{proof}\mbox{ } \\
\noindent {\bf Base case:} $[\ell]=1$.\\
In this case the empty set appears $\binom {N-1}{k}$ times and the singleton $\{1\}$ appears $\binom {N-1}{k-1}$ times. it is easy to see that this possible. Each ${\cal S}_i$ contains the empty set $\frac 1 M \binom {N-1}{k}$ times and $\{1\} $, $\frac 1 M \binom {N-1}{k-1}$ times.


\noindent {\bf Induction hypothesis:}  The statement holds true for $[\ell], \ell<N$. 

\noindent {\bf Induction step:}  True for $[\ell +1$]\\
Suppose the partition satisfying the statement for $\ell$ already exist. We will select some of the sets and add the element $\ell +1$ such that the statement still holds. These sets will be selected using the integrality of the network flow. We first construct a Network as follows. Construct a directed bipartite multi graph with partitions $A$ and $B$ as follows. The vertices in $A$ are the classes ${\cal S}_1$  to ${\cal S}_M$. The vertices in $B$ are the subsets of $[\ell]$ of size at most $k$. There are $r$ directed edges from ${\cal S}_i$ to the set $T$ if the set $T$ appears in ${\cal S}_i$ $r$-times. The capacity of each of these edges is 1. There are two other special vertices source $s$ and sink $t$. There is an edge from source $s$ to every vertex in $S_i$ with capacity $\frac L N$. Every vertex $T$ in $B$ is connected to the sink $t$ with capacity $\binom {N-1-\ell}{k-|T|-1}$. This network has integral flows and it has a min cut given by the edges out of source $s$ and it has another min cut that contains all edges which go to sink $t$. 
Capacities of both these cuts is 
$\frac {ML} {N} = \binom {N-1}{k-1}.$ 
By the integral version of the max flow min cut Theorem there is an integral flow where every edge out of source $s$ and every edge into sink $t$ is saturated. 
The edges between $A$ and $B$ will have flow value  $0$ or $1$. 
Now we construct the collection $\{{\cal S}'_i\}$ for $[\ell +1]$ given the collection $\{{\cal S}_i\}$ for $[\ell]$.

\begin{figure}
    \centering
        \includegraphics[height=4cm]{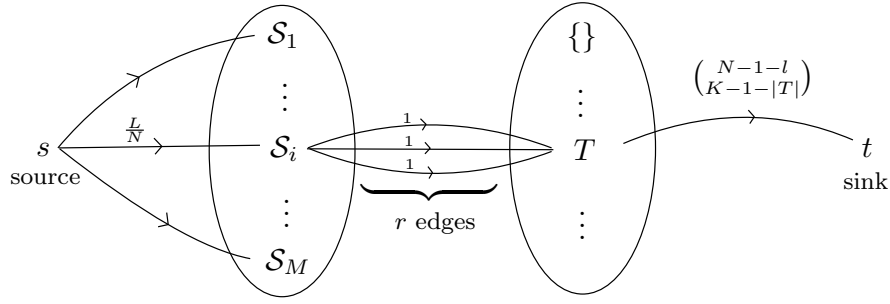} 
        \caption{ The flow network used to construct the hypergraph.  }
        \label{flow_hypergraph}
        \end{figure}

If $T$ is an element of ${\cal S}_i$ with multiplicity $r$ and $j$ of those edges have flow value $1$ then we add the element $\ell+1$ to $j$ of the sets $T$ in ${\cal S}_i$. This shows that the set $T\cup \{\ell+1\}$ appears in the collection 
$\{{\cal S}'_i\}$ $\binom {N-1-\ell}{k-1-|T|}$ times as 
required. Any set $T$ that does not contain $\ell +1$ 
appears 
$$\binom {N-\ell}{k-|T|}- \binom{N-\ell -1}{k-1-|T|}=\binom{N-\ell-1}{k-|T|} \mbox{ times.}$$ This proves the induction step and completes the proof.

\end{proof}

\begin{cor}\label{deg_hyp}
 $(d,d,\ldots,d)$ is the degree sequence of a $k$-regular hypergraph on vertex set $[N]$ if  and only if $k\mid dN$.
\end{cor}
\begin{proof} \mbox{ } \\
Necessity: Follows from double counting $(i,E)$, $i\in E$.\\
Sufficiency: 
 By Theorem~\ref{partition} for $\ell=N$ we get a partition ${\cal S}_1, {\cal S}_2,\ldots , {\cal S}_M \}$ 
of $\cal P$. 
 Since $d\mid kN$, $d$ is a multiple of $\frac L N$.  We take the edge set of the hypergraph as the union $\cup_{i=1}^{\frac {dN} { L } } {\cal S}_i$. In this hypergraph every vertex appears exactly in $d$ edges. 
\end{proof}

It may be noted that flow algorithm runs in polynomial time. Hence we can find the partition of Theorem~\ref{partition} in polynomial time  and also the required hypergraph in polynomial time.

The Corollary~\ref{deg_hyp} be considered as a  special case of $t-(v,k,\lambda)$ design. It would be interesting if the technique could be extended to include more general cases. The design problem was solved by  Keevash in \cite{keevash} using multiple ideas.

\section*{Acknowledgements:} We would like to thank Professor D. K. Ray-Chaudhuri for his inputs at the beginning of this project. We would also like to thank Professor Tarkeshwar Singh for his support.

\end{document}